\documentclass[10pt]{article}
\usepackage[english]{babel}
\usepackage{latexsym,amssymb,amsmath,amsthm,amsfonts}
\usepackage{lineno}
\usepackage{float,graphicx}

\usepackage{url}
\usepackage[explicit]{titlesec}
\usepackage[normalem]{ulem}
\usepackage{lipsum}

\newtheorem{theorem}{\uppercase{Theorem}}[section]
\newtheorem{corollary}{\uppercase{Corollary}}[section]
\newtheorem{conjecture}{\uppercase{Conjecture}}[section]
\newtheorem{lemma}{\uppercase{Lemma}}[section]
\theoremstyle{definition}
\newtheorem{definition}{\uppercase{Definition}}[section]

\titleformat{\section}
  {\normalfont\Large\bfseries\filcenter}{}{0em}{{\thesection.\hspace*{ .5 em}{#1}}}
\titleformat{name=\section,numberless}
  {\normalfont\Large\bfseries\filcenter}{}{0em}{{{#1}}}

\titleformat{\subsection}
  {\normalfont\large\bfseries}{}{0em}{{\thesubsection\hspace*{ 1 em}{#1}}}
\titleformat{name=\subsection,numberless}
  {\normalfont\Large\bfseries\filcenter}{}{0em}{{{#1}}}

\title{\uppercase{On crossing families of complete geometric graphs}}
\author{Dolores Lara $^1$ \and Christian Rubio-Montiel$^2$}

\date{$^1$Departamento de Computaci{\' o}n, Centro de Investigaci{\' o}n y de Estudios Avanzados del Instituto Polit{\' e}cnico Nacional, Mexico\\ e-mail: dlara@cs.cinvestav.mx\\
$^2$ Divisi{\' o}n de Matem{\' a}ticas e Ingenier{\' i}a, FES Acatl{\' a}n, Universidad Nacional Aut{\' o}noma de M{\' e}xico, 53150, Naucalpan, Mexico\\ e-mail: christian.rubio@apolo.acatlan.unam.mx\\[2ex]%
\today{}
}

\begin{document}
\maketitle

\makeatletter
\renewcommand\subsection{\@startsection{subsection}{3}{\z@}%
                                     {-3.25ex\@plus -1ex \@minus -.2ex}%
                                     {-1.5ex \@plus -.2ex}%
                                     {\normalfont\normalsize\bfseries}}
\makeatother

\begin{abstract}
A crossing family is a collection of pairwise crossing segments, this concept was introduced by Aronov et. al. (1994). They prove that any set of $n$ points (in general position) in the plain contains a crossing family of size $\sqrt{n/12}$. In this paper we present a generalization of the concept and give several results regarding this generalization. 
\end{abstract}


\section{Introduction}

A \emph{geometric graph} $\mathsf{G} = (\mathsf{V}, \mathsf{E})$ is an ordered pair of finite sets with the following properties: $\mathsf{V} \subseteq \mathbb{R}^2$, each edge is a line segment between two vertices, different edges have different sets of endpoints, and the interior of an edge contains no vertex. We call the elements of $\mathsf{V}$  \emph{points or vertices}, and the elements of $\mathsf{E}$  \emph{edges or segments}, indistinctly. A geometric graph is \emph{complete} if there is an edge between every pair of points of $\mathsf{V}$.  
Throughout this paper we assume that all sets of points in the plane are in general position: no three points are collinear. Note that any set $S$ of $n$ points in the plane induces a complete geometric graph. 

Let $S$ be a set of $n$ points in general position in the plane, and let $\mathsf{G} $ be a geometric graph. We say that  \emph{$\mathsf{G} $ is defined over $S$} if the vertex set of $\mathsf{G} $ is equal to $S$. An isomorphism between two graphs $G$ and $G'$ is a bijection $f: V \rightarrow  V'$ in which
$xy \in E(G)$ if and only if $f(x)f(y) \in E(G'), \forall x,y \in V$. A (rectilinear) \emph{drawing} of an (abstract) graph $G$ is an isomorphism between $G$ and a geometric graph $\mathsf{G}$ \cite{Diestel}. Let $\mathsf{G_1}$ and $\mathsf{G_2}$ be two (geometric) subgraphs of $\mathsf{G}$. We say that $\mathsf{G_1}$ and $\mathsf{G_2}$ \emph{cross} if there is one edge in $\mathsf{G_1}$ and one edge in $\mathsf{G_2}$ that have exactly one interior point in common. 

\begin{definition}[\emph{$H$-crossing family}]
Let $G$ be a graph and let $\mathsf{G}$ be a drawing of G. Given a subgraph $H$ of $G$, a family of vertex-disjoint (geometric) subgraphs of $\mathsf{G}$ are an \emph{$H$-crossing family} of $\mathsf{G}$ if (1) there exist an isomorphism between $H$ and every element of the family, and (2) are pairwise crossing. Equivalently we can say that $G$ has an $H$-crossing family of vertex disjoint isomorphic copies of $H$, if each two copies cross in the given drawing of $G$.
\end{definition}

The study of crossing families was introduced in \cite{MR1289067}, where the authors define a \emph{crossing family} as a collection of pairwise crossing segments. In current notation, their definition corresponds to a $K_2$-crossing family. They proved that any complete geometric graph on $n$ points has a $K_2$-crossing family of size $\sqrt{n/12}$. The authors of \cite{MR1289067} pointed out that the maximum size of such a crossing family could be even linear. There are several particular point set configurations \cite{MR1661386} with $K_2$-crossing families of linear size. However, the problem of finding, in any complete geometric graph, a $K_2$-crossing family having more than $O(\sqrt{n})$ elements is still open. The bound of $\sqrt{n/12}$ was proven using mutually avoiding sets of points. Two subsets, $R$ and $B$, of a given point set are \emph{mutually avoiding} if any line passing through two elements of $R$ does not intersect the convex hull of $B$, and viceversa. The author of \cite{MR1425226} prove that there exists an $n$-point set for which there are no mutually avoiding sets with more than $11\sqrt{n}$ elements. This result implies that the technique of using mutually avoiding sets to find $K_2$-crossing families can not be further extended to derive a linear bound. 

The authors of \cite{MR1755231} prove that every complete geometric graph contains a $2K_2$-crossing family of size $n/20$. A similar result can be deduced from \cite{MR1608874}. Recently, it was shown in \cite{rebollarcrossing} that every complete geometric graph contains a $K_3$-crossing family of size $n/6$, and also a $P_4$-crossing family of size $n/4$ ($P_k$ denotes the $k$-vertex path). This last bound is tight, since any $H$-crossing family has cardinality at most $n/|V(H)|$.

In this paper we present several results about crossing families. In Section~\ref{section2} we prove that every complete geometric graph contains a $P_3$-crossing family of size $O(\sqrt{n/2})$, a $K_{1,3}$-crossing family of size $n/6$, and a $K_{4}$-crossing family of size $n/4$ (which is tight). 

In Section~\ref{section3} we study, for small values of $n$, some numbers related to the Erd{\H o}s-Szekeres theorem. We define such numbers next. Let $f_{H}(k)$ be the smallest integer for which any complete geometric graph with $n=f_{H}(k)$ vertices, has an $H$-crossing family of size $k$. It is known that $f_{K_2}(3)=10$. That is, every set of $n\geq 10$ points in general position in the plane, has a $K_2$-crossing family with $3$ elements. Furthermore, there are complete geometric graphs of size $n=9$, having $K_2$-crossing families with at most $2$ elements. For these results see \cite{O:2016:Online,MR1942187,nielsen2013some}. We prove that $f_{P_3}(3)=9$ and $f_{K_3}(3)=9$.

In Section~\ref{section4} we extend the notion of crossing families to intersecting families:  $G$ has an \emph{$H$-intersecting family} of isomorphic copies of $H$, if in the given drawing of $G$, each two copies are edge-disjoint and cross. Note that, in this case, we remove the condition which ask for the subgraphs belonging to the family to be vertex disjoint. Several authors have studied the $K_2$-intersecting family before under the name of ``straight-line thrackle''. Erd{\H o}s prove in \cite{MR0015796} that for any complete geometric graph, every $K_2$-intersecting family has at most $n$ edges. The lower bound of $n-1$ is easy to verify. 
We prove that every complete geometric graph contains a $P_3$-intersecting family of size $\frac{n^{3/2}}{12\sqrt{6}}$. We present some corollaries about other intersecting families, including results about families in complete balanced bipartite geometric graphs. We close with two conjectures.

\section{Crossing families}\label{section2}

Before we continue the discussion, let us introduce some useful definitions and results.
Let $R$ and $B$ be two sets of $w$ points that can be separated by a line $\ell$.
Given a point $b \in B$, sort the elements of $R$ in counterclockwise order around $b$. Let the sorted set be $\{r_1, \ldots, r_w\}$. We now give three definitions. First, we say that $b$ \emph{sees $r_i$ at rank $i$}. Additionally, we say that $r_i$ obeys the \emph{rank condition from $B$} if all points in $B$ see $r_i$ at rank $i$. Lastly, we say that $R$ obeys the \emph{ rank condition from B} if there exist a labeling $\{r_1, \ldots, r_w\}$ for $R$, such that for each $i, 1 \leq i \leq w$, $r_i$ obeys the rank condition from $B$. $R$ \emph{avoids} $B$ is any line passing through two points of $R$ does not intersect the convex hull of $B$. The authors of \cite{MR1289067} proved that $R$ avoids $B$ if and only if $R$ has the rank condition from $B$. We will use this result in the proof of Theorem~\ref{theorem1} and Theorem~\ref{theorem2}, as well as the following lemma originally given by Erd{\H o}s and Szekeres \cite{MR1556929}.

\begin{lemma}[Erd{\H o}s and Szekeres \cite{MR1556929}]\label{lemma1}
Any sequence of $n$ real numbers contains either an ascending or a descending subsequence of length $\sqrt{n}$.
\end{lemma}

The following lemma, proved by Mediggo \cite{Megiddo}, will be used in the proof of Theorem~\ref{theorem4}.

\begin{lemma}[Mediggo \cite{Megiddo}]\label{lemma2}
Let $S$ be a set of $n$ points in the plane, for any given line $\ell_1$ dividing $S$, it is possible to find another line $\ell_2$ which simultaneously divide the points in both parts of $\ell_1$ in any desired proportions.
\end{lemma}

In several proves we divide the plane using a specific configuration of lines, in our constructions we use the following theorem originally prove by Ceder \cite{MR0188890}.

\begin{theorem}[Ceder \cite{MR0188890}]\label{thm:ceder}
Let $\mu$ be a finite measure absolutely continuos with respect to the Lebesgue measure on $\mathbb{R}^2$. There are three concurrent lines that partition the plane into six parts of equal measure.
\end{theorem}

From this theorem the authors of \cite{HLR} proved the following result. 

\begin{corollary}[\cite{HLR}]\label{thm:partition}
Let $S$ be a set of $n$ points in general position in the plane. There exist three lines, two of them parallel, which divide the set into six parts with at least $\frac{n}{6}-1$ points each.

\end{corollary}
We divide our discussion into three subsections: one for the $P_3$-crossing families, one for the $K_{1,t}$-crossing families, and one for the $K_{t}$-crossing families.

\subsection{$P_3$-crossing families}

Let $\mathsf{K}_n$ be a complete geometric graph with $n$ points. In this subsection we show that there exist a $P_3$-crossing family of $\mathsf{K}_n$ of size $O(\sqrt{n/2})$. In order to construct the family, first we divide the plane into six parts, and then we use these parts to conveniently choose three subsets of vertices of $\mathsf{K}_n$. Lastly, we show that there exist a $P_3$-crossing family that has a vertex in each one of the subsets. 

\begin{theorem}\label{theorem1}
Let $S$ be a set of $n$ points in general position in the plane, and let $\mathsf{K}_n$ be the complete geometric graph defined over $S$.  $\mathsf{K}_n$ contains a $P_3$-crossing family of size at least $\sqrt{n/2+1}-1$.
\end{theorem}
\begin{proof}
Let $S$ be a set of $n$ points in general position in the plane, and let $\mathsf{K}_n $ be the complete geometric graph defined over $S$. First we divide the plane into six regions, and then we use these regions to choose a subset of the edges of $\mathsf{K}_n$ which we prove is a $P_3$-crossing family.

We divide the plane using a specific configuration of lines as follows. Applying an affine transformation, and by a simple adjustment of the proof of Corollary~\ref{thm:partition}, there are two vertical lines $\ell_1$ and $\ell_2$, and one horizontal line $\ell_3$, so that they divide the set $S$ into $6$ subsets $S_1, S_2, S_3, S_4, S_5$ and $S_6$, (refer to Figure~\ref{Fig1}), such that $|S_1|=|S_3|=|S_4|=|S_6|=w=\sqrt{n/2+1}-1$ and $|S_2|=n/2-2w$.

\begin{figure}
\begin{center}
\includegraphics[scale=0.85]{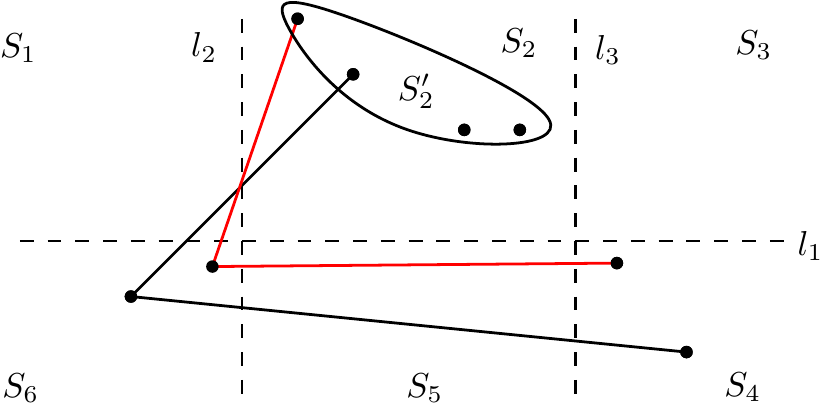}
\caption{The line configuration used in the proof of Theorem~\ref{theorem1}}\label{Fig1}
\end{center}
\end{figure}

Consider the sequence obtained by sorting the points in $S_2$ according to their $x$-coordinate from left to right. Lemma~\ref{lemma1} implies that this sequence contains a subsequence with $y$-coordinates either in increasing or in decreasing order. Let the set of points in this subsequence be $S'_2 \subset S_2$. Suppose without loss of generality that the sequence is decreasing, label its points as $\{x_1,\dots,x_w\}$. Observe that for every $i \in \{1, \ldots, w \}$, $x_i$ has the rank condition from $S_6$. Hence $S'_2$ has the rank condition from $S_6$. Therefore, $S'_2$ avoids $S_6$.

Let $M$ be a set of vertex-disjoint edges joining every point of $S_4$ to some point of $S_6$.
That is, every point in $S_4$ must appear in exactly one edge in $M$. Since we require every two edges in $M$ to be vertex-disjoint, then $M$ has size $w$; furthermore each member of $M$ crosses $\ell_2$. Let $M = \{ m_1, \ldots, m_w\}$. For each $i \in \{ 1, \ldots, w\}$, label the extremes of $m_i$ as $y_i$ and $z_i$, where $y_i \in S_4$ and $z_i \in S_6$.  

To end this proof, consider the complete geometric graph $\mathsf{K}_n$ induced by the set $S$. Note that the subgraphs of $\mathsf{K}_n$ with vertex set $\{x_i,y_i,z_i\}$ and edge set $\{x_iy_i,y_iz_i\}$ are $3$-vertex paths. Furthermore, each pair of these subgraphs intersects. Therefore, $\mathsf{K}_n$  contains a $P_3$-crossing family of size $w = \sqrt{n/2+1}-1$. 
\end{proof}

\subsection{$K_{1,t}$-crossing families}

Let $\mathsf{K}_n$ be a complete geometric graph with $n$ points. In this subsection we prove that there exist a $K_{1,3}$-crossing family of $\mathsf{K}_n$ with $\frac{n}{6}$ elements. As in the previous subsection we construct the crossing family by dividing the plane into six parts, and then we use the set of points induced by these regions to carefully choose the elements of the $K_{1,3}$-crossing family. From this result we derive the existence of a $K_{1,t}$-crossing family of linear size, for any $t\geq 3$. 

\begin{theorem}\label{theorem2}

Let $S$ be a set of $n$ points in general position in the plane, and let  $\mathsf{K}_n$ be the complete geometric graph defined over $S$. $\mathsf{K}_n$ contains a $K_{1,3}$-crossing family of size $\frac{n}{6} -1 $.
\end{theorem}
\begin{proof}
Corollary~\ref{thm:partition} implies that there exist three lines, two of them parallel, which divide the plane into six regions; each region with at least $\frac{n}{6}-1$ points of $S$ in its interior. We label the subsets of $S$ induced by each one of these regions as $S_1, S_2, \ldots, S_6$; refer to Figure~\ref{Fig2}.  If any of these sets has cardinality bigger than $\frac{n}{6} -1$ remove the exceeding points, choose them arbitrary.

\begin{figure}
\begin{center}
\includegraphics[scale=0.85]{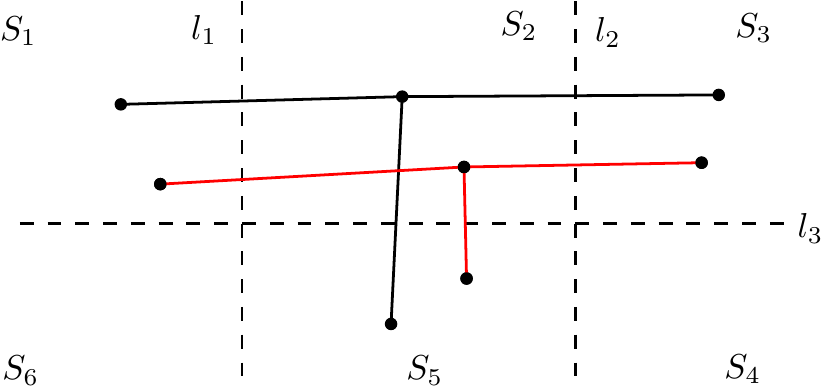}
\caption{ The line configuration used in the proof of Theorem~\ref{theorem2}.}\label{Fig2}
\end{center}
\end{figure}

Now, let $S_1=\{x_1,\dots,x_{{n/6}-1}\}$, $S_2=\{y_1,\dots,y_{{n/6}-1}\}$, $S_3=\{z_1,\dots,z_{{n/6}-1}\}$ and $S_5=\{w_1,\dots,w_{{n/6}-1}\}$. For $i \in \{1, \ldots, {\frac{n}{6}}-1\}$, the subgraphs with vertex set $\{x_i,y_i,z_i,w_i\}$ and edge set $\{x_iy_i,y_iz_i,y_iw_i\} $,  are a $K_{1,3}$-crossing family of $\mathsf{K}_n$ with $\frac{n}{6}-1$ elements.
\end{proof}

The technique used in the proof of the above theorem can be easily adapted to generate $K_{1,t}$-crossing families, with $t \geq 4$. For $t=4$, using the same construction as before, construct the $K_{1,3}$-crossing family of size $\frac{n}{6}-1$, then join each graph in the family with some other point. Since the graphs that constitute the crossing family must be vertex disjoint, for $t \geq 5$,  it is necessary to first choose
$\frac{6n}{t+1}$ points, and then use these points to construct the $K_{1,3}$-crossing family.  Once the $K_{1,3}$-crossing family has been constructed, join each of these graphs with $t-3$ points in the complement of the chosen set. This procedure generates a  $K_{1,t}$-crossing family of size $n/(t+1)$. Therefore, the next two corollaries follow.  

\begin{corollary}
$\mathsf{K}_n$ contains a $K_{1,4}$-crossing family of size $\frac{n}{6}-1$.
\end{corollary}

\begin{corollary}
$\mathsf{K}_n$ contains a $K_{1,t}$-crossing family of size $n/(t+1)$, for all $t\geq 5$.
\end{corollary}

\subsection{$K_t$-crossing families}
 Let $\mathsf{K}_n$ be a complete geometric graph with $n$ points. 
In this subsection we show that here exist a $K_{4}$-crossing family of $\mathsf{K}_n$ with $O(\frac{n}{4})$ elements. As in the previous section, in order to construct the crossing family, first we divide the plane into seven regions, and then we use these parts to conveniently choose four subsets of vertices of  $\mathsf{K}_n$. Lastly, we show that there exist a $K_4$-crossing family that has a vertex in each one of the subsets. From this result we derive the existence of a $K_{t}$-crossing family with $\left \lfloor \frac{n}{t} \right \rfloor$ elements, for any $t\geq 4$. We would like to point out that this result can be deduced from the one presented in \cite{rebollarcrossing} about $P_4$-crossing families, however our proof is different and allow us to imply a corollary. 

\begin{theorem}\label{thm}
Let $S$ be a set of $n$ points in general position in the plane, and let $\mathsf{K}_n$ be the complete geometric graph defined over $S$.
 $\mathsf{K}_n$ contains a $K_{4}$-crossing family of size at least $\frac{n}{4} - 6$.
\end{theorem}
\begin{proof}

Let $\ell_1$ be an horizontal line dividing the set $S$ into two subsets: one of size $ \left \lceil \frac{n}{4} \right \rceil$ and the other one of size $\left \lfloor \frac{3n}{4} \right \rfloor$. Label the set with the smallest cardinality as $S_1$. Theorem~\ref{thm:ceder} implies that there are three concurrent lines that divide the set $S \setminus S_1$ into $6$ parts, each containing at least $\frac{n}{8} - 1$  points in its interior. If any of these sets has cardinality bigger than $\frac{n}{8} -1$ remove the exceeding points, choose them arbitrary.

\begin{figure}
\begin{center}
\includegraphics[scale=0.85]{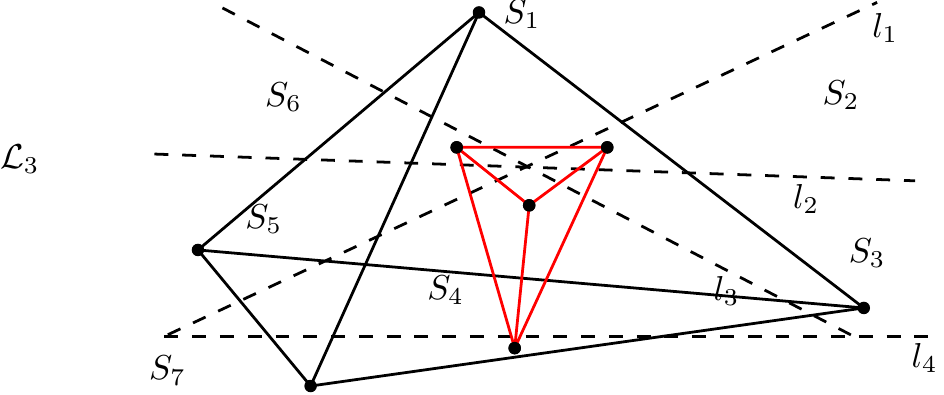}
\caption{The line configuration used in the proof of Theorem~\ref{thm}}\label{Fig3}
\end{center}
\end{figure}

Label any $2\left(\frac{n}{8}-1\right)$ points of $S_1$ as $\{x^1_1,\dots,x^1_{n/8-1},x^8_1,\dots,x^8_{n/8-1}\}$. For $i\in\{2,\dots,7\}$, let $S_i=\{x^i_1,\dots,x^i_{n/8-1}\}$. Now we show that the complete subgraphs $T_{i,odd}$ with vertex set $\{x^1_i,x^3_i,x^5_i,x^7_i\}$ and the complete subgraphs $T_{i,even}$ with vertex set $\{x^2_i,x^4_i,x^6_i, x^8_i\}$ are a $K_{4}$-crossing family. Each of these graphs can be expressed as the join of $K_3 + K_1$, where the only vertex of $K_1$ is a point of $S_1$. Let $p$ be the intersection point of the lines $\ell_2, \ell_3$ and $\ell_4$; note that $p$ is contained in the interior of each $\mathsf{K_3}$.
Suppose, by contradiction, that there are two subgraphs $\mathsf{A = K_3^A + K_1^A}$ and $\mathsf{B = K_3^B + K_1^B}$ in the set $T_{i,odd} \cup T_{i,even}$ that do not intersect each other. Since in addition of not intersecting, both $\mathsf{K_3^A}$ and $\mathsf{K_3^B}$ contain $p$, then without loss of generality $\mathsf{K_3^B}$ is inside $\mathsf{K_3^A}$. Three edges of $B$  connect $\mathsf{K_3^B}$ (in the interior of $\mathsf{K_3^A}$) with a point in $S_1$ (in the exterior of $\mathsf{K_3^A}$) and therefore these edges intersect at least of edge of $A$. This is a contradiction, and therefore the complete subgraphs subgraphs $T_{i,odd}$ and  $T_{i,even}$ are a $K_{4}$-crossing family with at least $\frac{n}{4} -6 $ elements.
\end{proof}

\begin{corollary}
$\mathsf{K}_n$ contains a $K_{t}$-crossing family of size $n/t$ for all $t\geq 4$.
\end{corollary}


\section{Small numbers related to the {E}rd{\H o}s-{S}zekeres Theorem}\label{section3}

For every natural number $k\geq 3$, let $f(k)$ be the least number such that every planar set of $f(k)$ points, in general position, contains the vertices of some convex $k$-gon. This parameter was introduced by Erd{\H o}s and Szekeres \cite{MR1556929}. It is known that $1+2^{k-2}\leq f(k)\leq \binom{2k-4}{k-2}+1$, see \cite{MR2178339}. In fact, Erd{\H o}s and Szekeres conjectured in their paper that the lower bound is in fact an equality. In a recent paper \cite{ASuk17} it was shown that $f(k)=2^{k+o(k)}$.

Table~\ref{exact4} shows the first values of $f(k)$, these values were taken from \cite{MR2291511}.
\begin{table}[!htbp]
\begin{center}
\begin{tabular}{|c|cccc|}
\hline \hline
$k$ & 3 & 4 & 5 & 6 \\
\hline
$f(k)$ & 3 & 5 & 9 & 17 \\
\hline
\end{tabular}
\caption{\label{exact4}Exact values for $f(k)$, $3\leq k\leq 6.$}
\end{center}
\end{table}

Given the apparent difficulty of determining values of $f(n)$, the authors of  \cite{nielsen2013some} propose to study a new function. Let $f_{H}(k)$ be the least integer such that any complete geometric graph with $f_{H}(k)$ vertices contains an $H$-crossing family of size $k$. 

Note that for every $n\geq 2k$ the vertex set of a convex $n$-gon generates a crossing family of size $k$, then $f_{K_2}(k)\leq f(k)$. Moreover, the result given in \cite{MR1289067} proves that $f_{K_2}(k)\leq 12k^2$.

It is known that $f_{K_2}(3)=10$, that is, every geometric graph with at least  $10$ vertices contains a $K_2$-crossing family of $3$ elements. However, there are only twelve different complete geometric graphs with $9$ vertices that have $K_2$-crossing families of size at most $2$, all others contain a $K_2$-crossing family of $3$ elements, see \cite{O:2016:Online,MR1942187,nielsen2013some}. Now we prove that $f_{P_3}(3)=9$. Theorem~\ref{theorem1} establish the upper bound $f_{P_3}(k)\leq 2(k+1)^2-2$, however for the particular case of $k=3$, a better bound is given by the facts that $f_{P_3}(k)\leq f_{K_2}(k)$ and $f_{K_2}(3) = 10$. Then, $f_{P_3}(3)\leq10$. Now since every $f_{P_3}(3)$-set must be of size at least $9$, we have that $9 \leq f_{P_3}(3) \leq 10$. In order to prove that $f_{P_3}(3)=9$, it is sufficient to show that every set of $9$ points contains a $P_3$-crossing family of size $3$. Notice that every set of points containing a $K_2$-crossing family of size $3$ contains a $P_3$-crossing family of size $3$. As we say before, there are only twelve combinatorially different $9$-point sets that do not contain a $K_2$-crossing family of size $3$, therefore we must show, only for these sets, that they contain a $P_3$-crossing family of size $3$. It is not hard to find these desired families for the twelve point sets, we present them, due to lack of space, in the arXiv version of this paper \cite{LaraRubio:arxiv}, see also a draft version in \cite{LaraRubio:Online}. Therefore,  $f_{P_3}(3)=9$.

The fact that $f_{P_3}(3)=9$ implies that $f_{K_3}(3)=9$ since $f_{K_3}(k)\leq f_{P_3}(k)$.


\section{Intersecting families}\label{section4}

In this section we study intersecting families. 

\begin{definition}[\emph{$H$-intersecting family}]
 Let $G$ be a graph and let $\mathsf{G}$ be a drawing of G. Given a subgraph $H$ of $G$, a family of (geometric) subgraphs of $\mathsf{G}$ are an \emph{$H$-intersecting family} of $\mathsf{G}$ if (1) there exist an isomorphism between $H$ and every element of the family, (2) are edge-disjoint, and (3) are pairwise intersecting. Equivalently we can say that $G$ has an $H$-intersecting family of isomorphic copies of $H$, if each two copies are edge-disjoint and cross, in the given drawing of $G$.
\end{definition} 

Note that any $H$-intersecting family has cardinality at most $\binom{n}{2}/|E(H)|$. In this section we prove that for every complete geometric graph $\mathsf{K}_n$, and for every complete bipartite geometric graph $\mathsf{K}_{n/2, n/2}$ there exist a $P_3$-intersecting family of size $O(n^{3/2})$.

In order to construct the crossing family, first we divide the plane into six regions, and then we use these parts to conveniently choose two subsets of vertices of  $\mathsf{K}_n$. Lastly, we show that there exist a $P_3$-intersecting family that has a vertex in one of the subsets and two vertices in the other. 

\begin{theorem}\label{theorem4}
Let $S$ be a set of $n$ points in general position in the plane, and let $\mathsf{K}_n$ be the complete geometric graph defined over $S$. Let $S' = R \cup B$ be a set of $n$ points in general position in the plane, such that $R \cap B = \emptyset$ and $|R|=|B|$, and let $\mathsf{K}_{n/2,n/2}$ 
be the complete geometric graph defined over $S'$.
\begin{enumerate}
\item $\mathsf{K}_{n/2,n/2}$ contains a $P_3$-intersecting family of size at least $\frac{n^{3/2}}{24\sqrt{12}}$.
\item $\mathsf{K}_n$ contains a $P_3$-intersecting family of size at least $\frac{n^{3/2}}{12\sqrt{6}}$.
\end{enumerate}
\end{theorem}
\begin{proof}

Let $\ell_1$ be an horizontal line dividing the set $S'$ into two subsets so that $n/4$ points of one color are on one side of the line, and $n/4$ points of the other color are on the other side of the line. This can be achieve by moving $\ell_1$ down from $y=+\infty$ until $n/4$ on the first color, say red, are above the line. Discard the blue points above and the red points below $\ell_1$. Lemma~\ref{lemma2} implies that there exist a line $\ell_2$, which intersects $\ell_1$, and with exactly $n/12$ red points on its left and exactly $n/12$ blue points on its right. Consider a line $\ell_3$ parallel to $\ell_2$ and to its left. Label, in clockwise order, the six regions induced by the three lines as $R_1, \ldots, R_6$; refer to Figure~\ref{Fig4}. Notice that we can choose $\ell_3$ in such a way that there are at least $n/12$ red points in $R_2$, and either $n/12$ blue points in $R_6$ or $n/12$ blue points in $R_4$. Without loss of generality we can assume that the first case occurs. Let $S_i = R_i \cup S$, for $ 1 \leq i \leq 6$. The set $S_2$ has at least $n/12$ red points and the set $S_6$ has $n/12$ blue points. We can apply an affine transformation to make  $\ell_2$ and $\ell_3$ vertical.

\begin{figure}
\begin{center}
\includegraphics[scale=0.85]{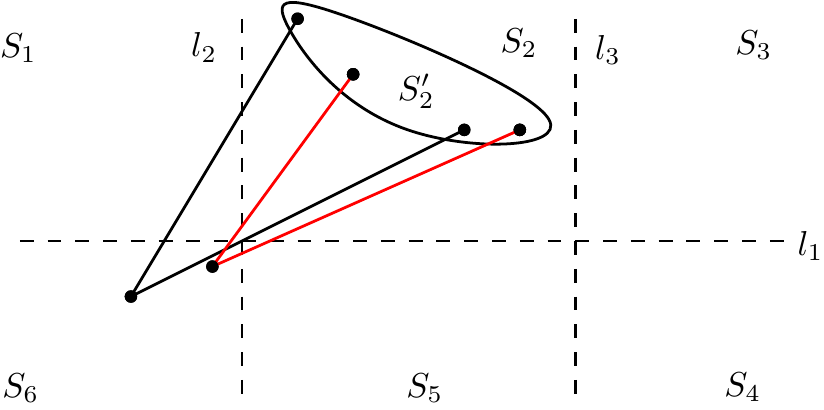}
\caption{The line configuration used in the proof of Theorem~\ref{theorem4}.}\label{Fig4}
\end{center}
\end{figure}

Consider the sequence obtained by sorting the points in $S_2$ according to their $x$-coordinate from left to right. Lemma~\ref{lemma1} implies that this sequence contains a subsequence with $y$-coordinates either in ascending or in descending order. Call the set of points in this subsequence $S'_2$. $S'_2$ has cardinality at least $w=\sqrt{n/12}$. Without loss of generality assume that the sequence is descending, let the sequence be $S'_2=\{x_1,\dots,x_{w/2},x'_1,\dots,x'_{w/2}\}$.  Observe that $S'_2$ avoids $S_6$ since $S'_2$ has the rank condition from $S_6$, that is, for all $i$, $x_i$ has the rank condition from $S_6$. Let $S_6=\{y_1,\dots,y_{n/12}\}$. The subgraphs $T_{i,j}$ with vertex set $\{x_i,y_j,x'_i\}$ and edge set $\{x_iy_j,y_jx'_i\}$, are a $P_3$-crossing family of $\mathsf{K}_{n/2,n/2}$ with $(w/2)(n/12)$ elements.

The only change for the uncolored case is that $\ell_1$ can be chosen without discarding half of the points.
\end{proof}

\begin{corollary}
$\mathsf{K}_n$ contains a $K_3$-intersecting family of size at least $\frac{n^{3/2}}{12\sqrt{6}}$.
\end{corollary}

Using a partition similar to the one used in Theorem \ref{theorem1} and subgraphs similar to the ones used in Theorem \ref{theorem4}, we obtain the following result. We omit the proof.

\begin{corollary}
$\mathsf{K}_{n/2,n/2}$ contains a $P_3$-crossing family of size at least $(\sqrt{n+1}-1)/4$.
\end{corollary}

From the proof of Theorem \ref{theorem2} it follows the following corollary.

\begin{corollary}
$\mathsf{K}_n$ contains a $K_{1,t}$-intersecting family of size at least $\frac{n^2}{36}$ for $t\in\{3,4,5\}$.
\end{corollary}

Given $k$ sets $A_1, \ldots, A_k$, of points in general position in the plane, a set $\mathcal{T} = \{p_1, \ldots, p_k\}$ is called a traversal of the $k$ sets $A_i$, if $p_1 \in A_1, \ldots, p_k \in A_k$. $\mathcal{T}$ is in convex position if each $p_i$ is a vertex of the convex hull of $\mathcal{T}$. The authors of  \cite{MR1755231} proved that every set of $n$ points in general position in the plane contains four subsets $A_1, \ldots, A_4$, each one with cardinality at least $n/20$, and such that every transversal of the sets $A_i$ is convex. For a similar theorem see \cite{MR1608874}. From this result it follows that any $K_n$ contains a $2K_2$-intersecting family of  size at least $n^2/400$. 

\begin{corollary}
$K_n$ contains a $P_4$-intersecting family of size at least $\frac{n^2}{400}$.
\end{corollary}

In \cite{HLR} it was proven that any $K_n$ contains a $K_4$-intersecting family of  size $n^2/24.5$. To conclude, we propose the following two conjectures.

\begin{conjecture}
$K_n$ contains a $P_3$-intersecting family of quadratic size.
\end{conjecture}

\begin{conjecture}
$K_n$ contains a $K_3$-intersecting family of quadratic size.
\end{conjecture}

\bibliographystyle{amsplain}
\bibliography{biblio}

\appendix
\section{Appendix}\label{apen}

In this appendix, we present the twelve  combinatorially different complete geometric graphs with nine vertices that contain a $P_3$-crossing family of size 3. In each graph the $3$-paths in the crossing family are shown in green, red and blue. The information about the twelve combinatorially different point sets (and therefore twelve combinatorially different complete geometric graphs) was taken from \cite{O:2016:Online}.

\begin{figure}[h!]
\begin{center}
\includegraphics[scale=0.2]{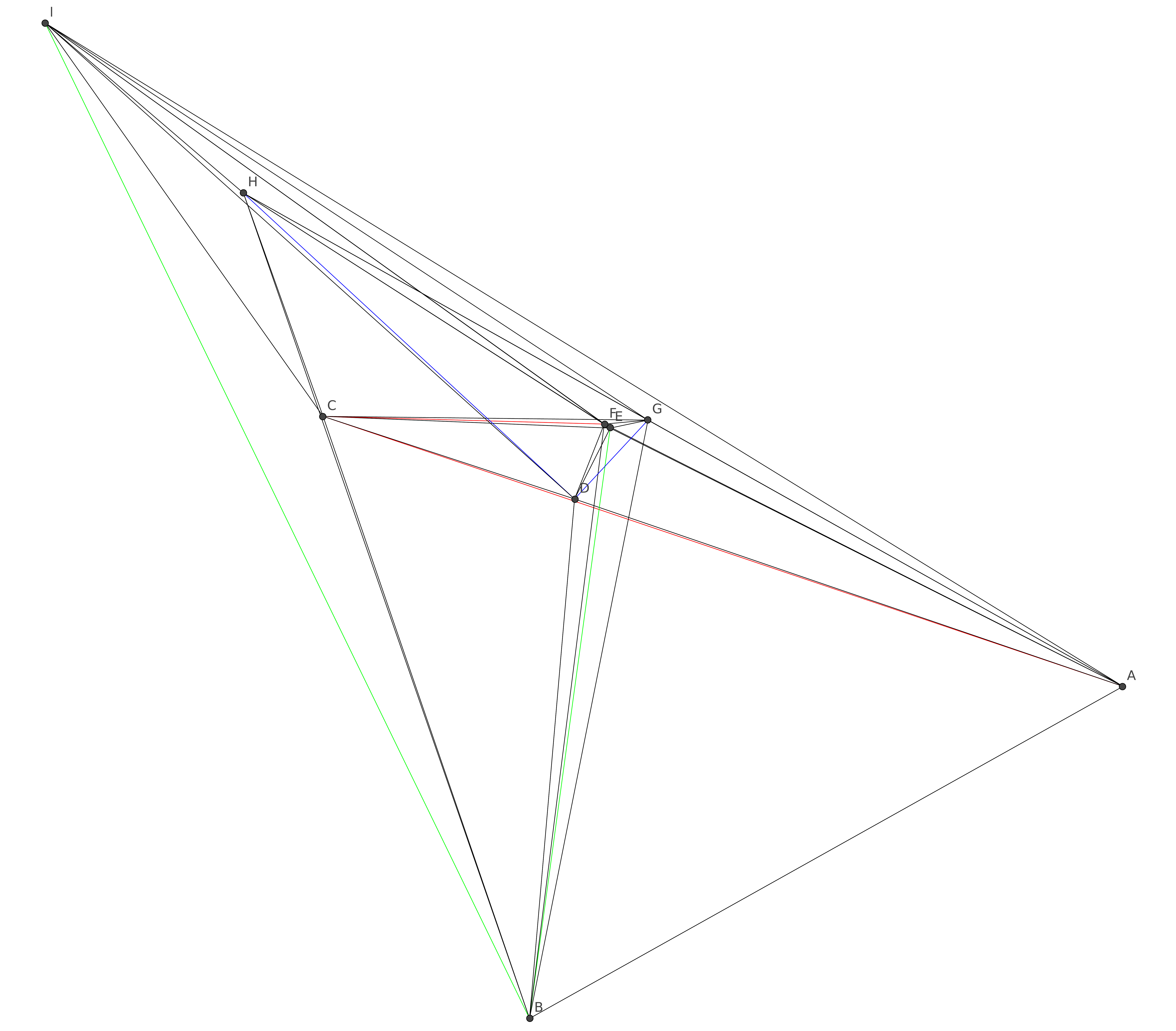}
\end{center}
\end{figure}

\begin{figure}
\begin{center}
\includegraphics[scale=0.27]{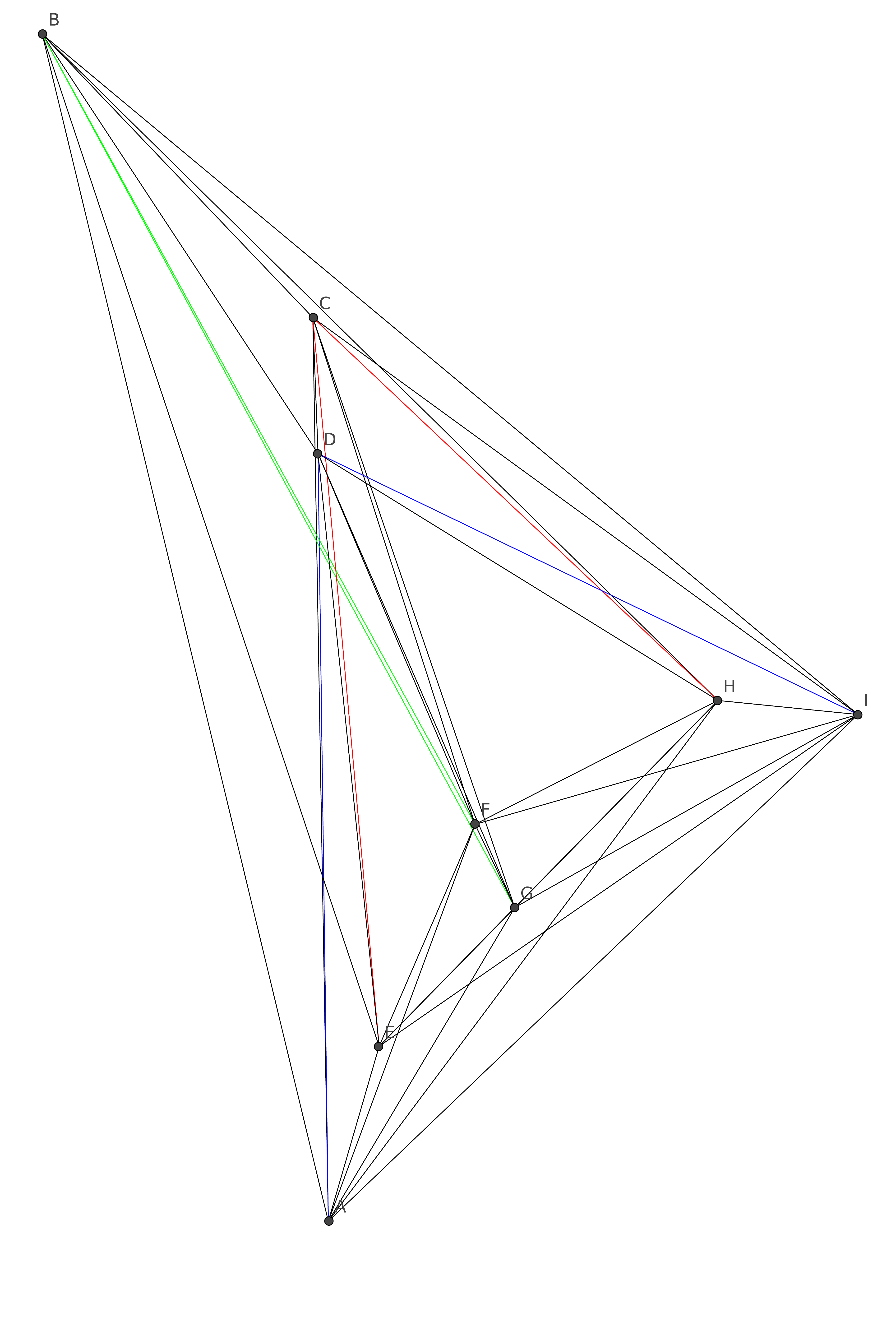}
\end{center}
\end{figure}

\begin{figure}
\begin{center}
\includegraphics[scale=0.23]{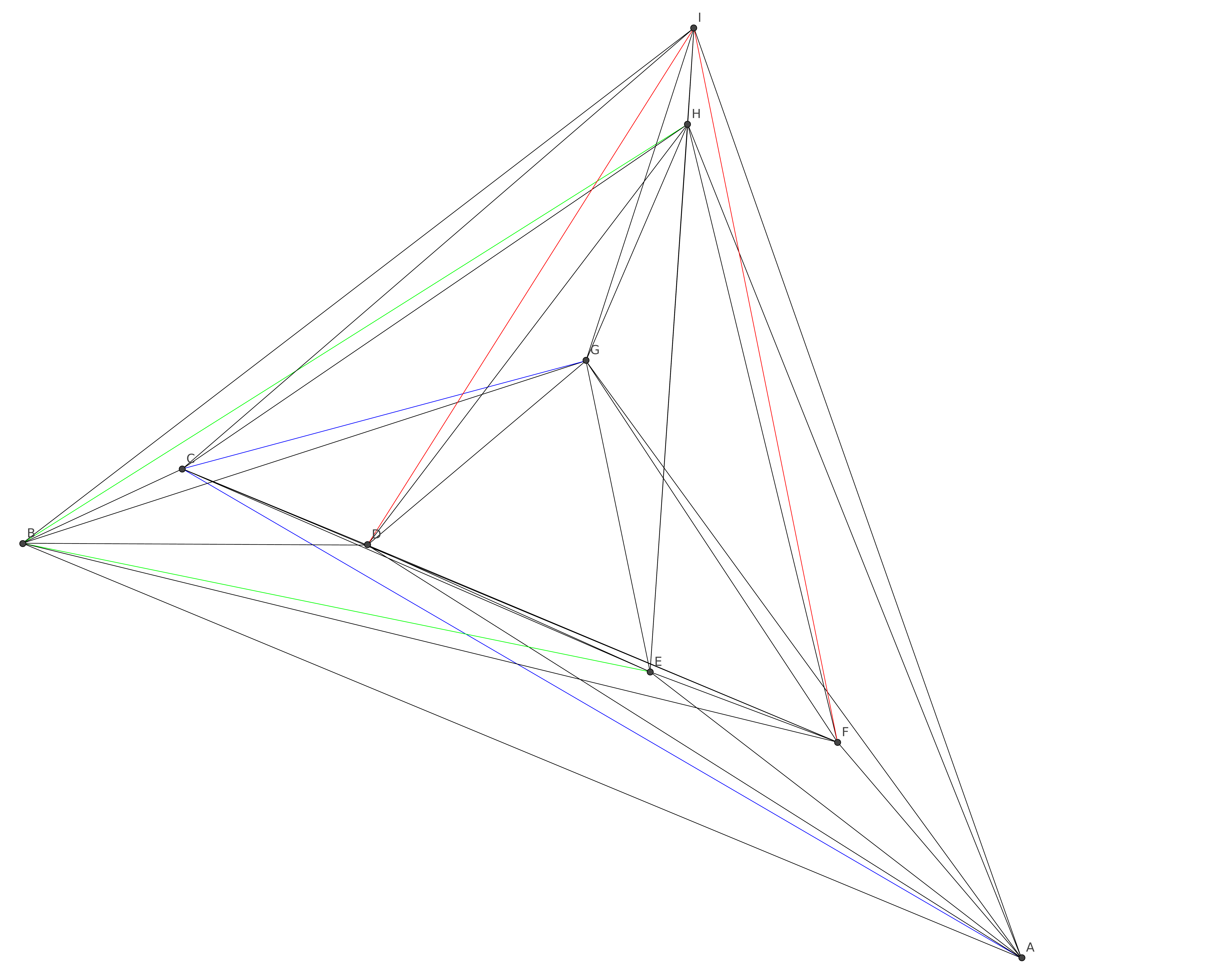}
\end{center}
\end{figure}

\begin{figure}
\begin{center}
\includegraphics[scale=0.23]{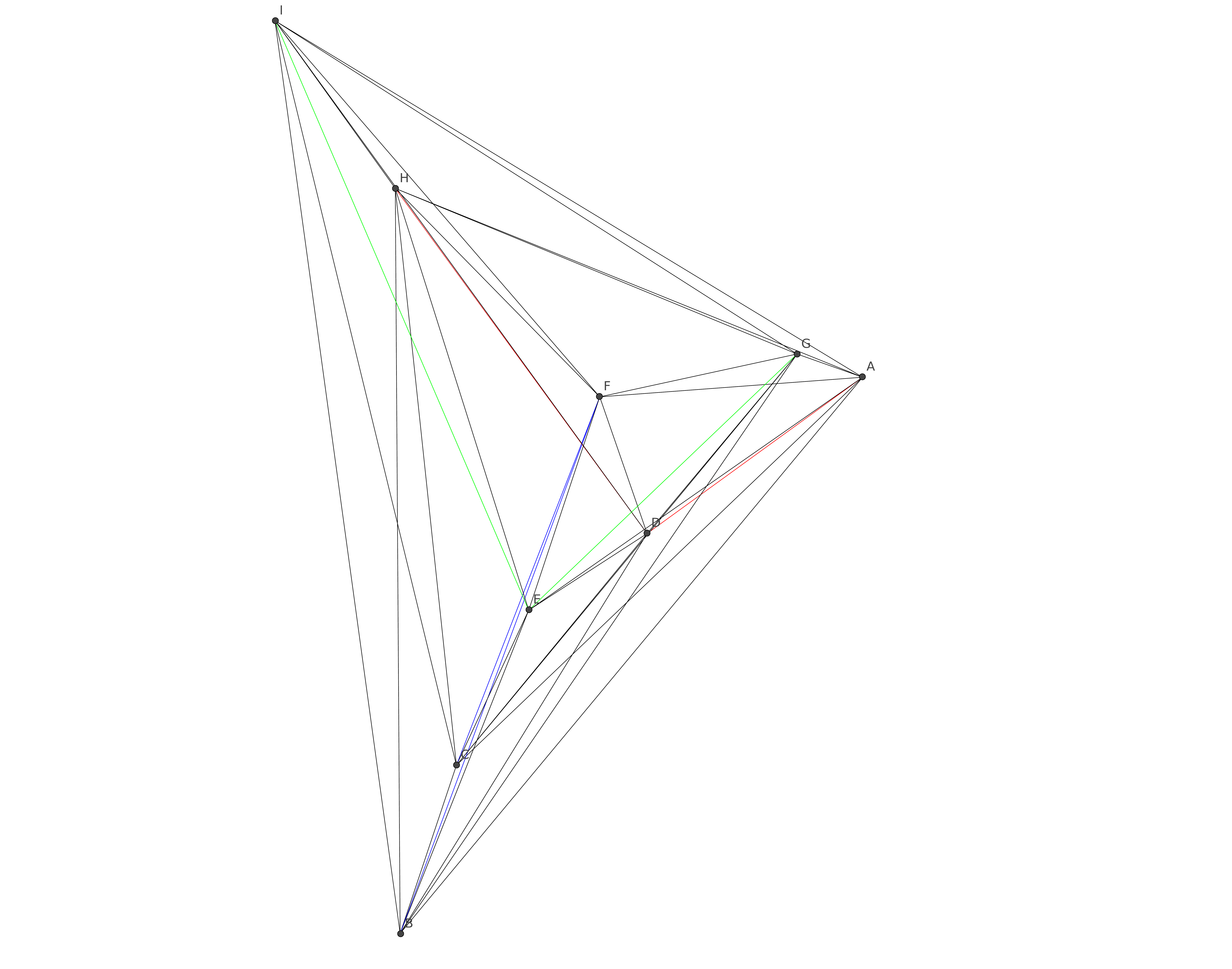}
\end{center}
\end{figure}

\begin{figure}
\begin{center}
\includegraphics[scale=0.23]{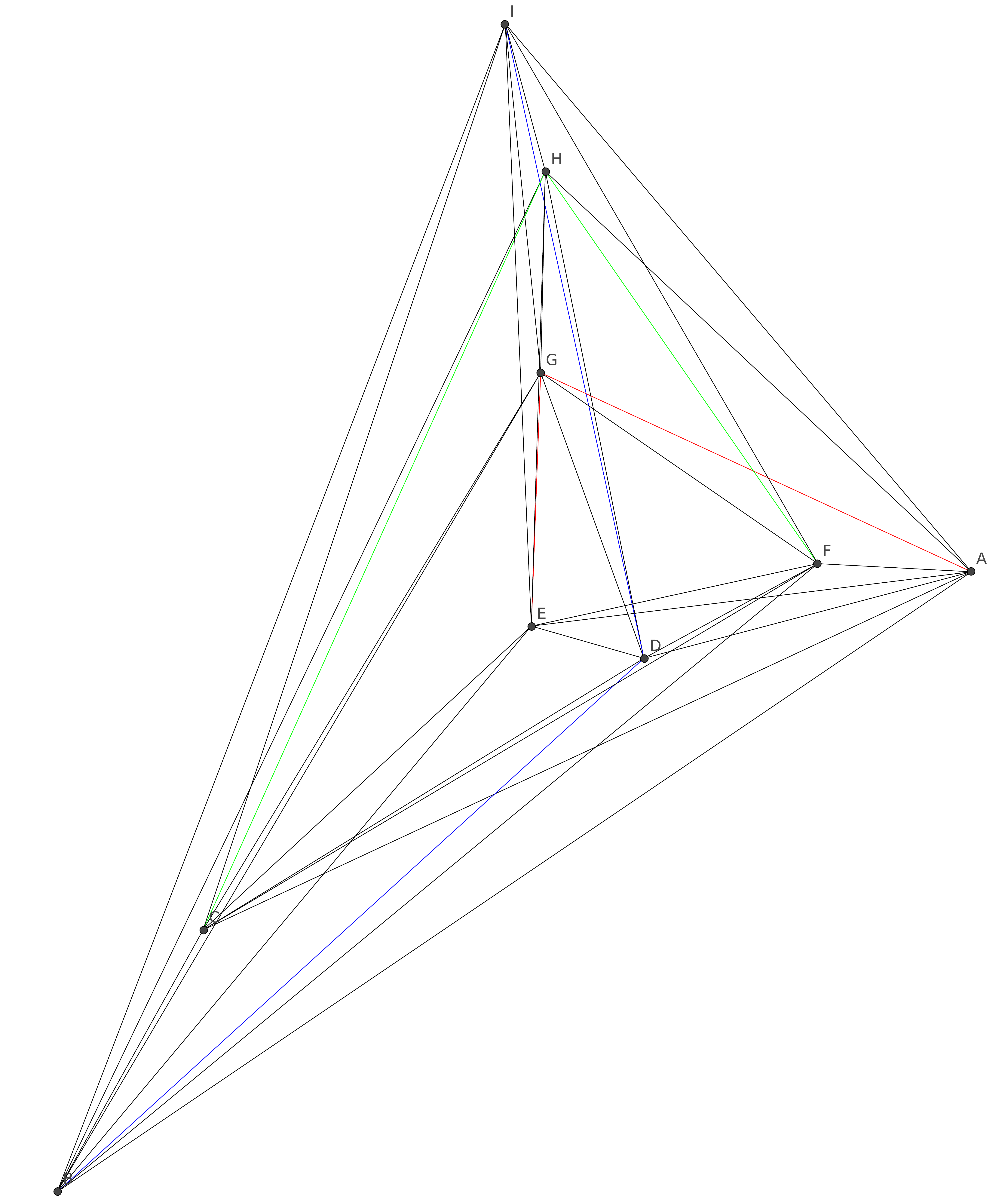}
\end{center}
\end{figure}

\begin{figure}
\begin{center}
\includegraphics[scale=0.23]{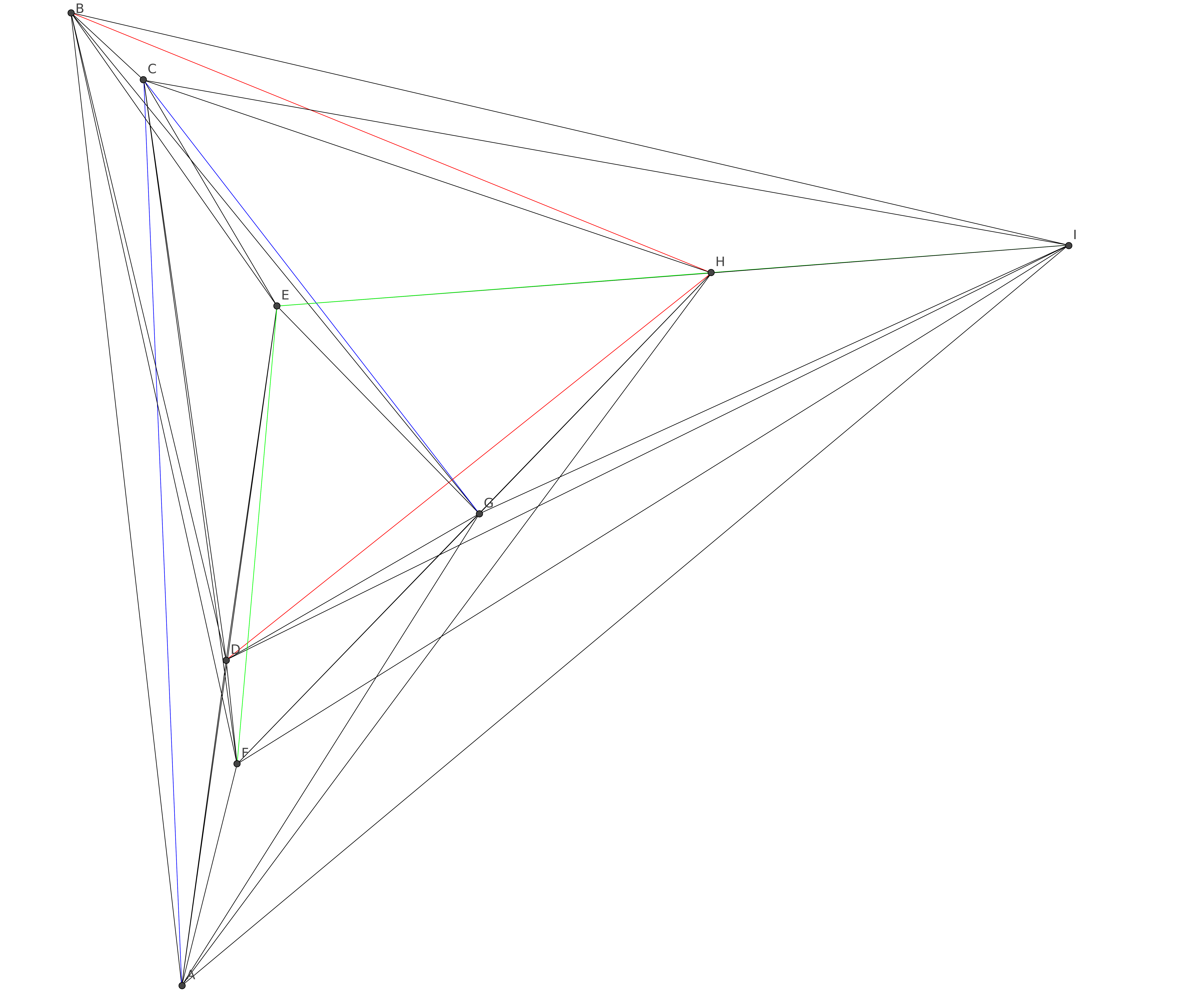}
\end{center}
\end{figure}

\begin{figure}
\begin{center}
\includegraphics[scale=0.23]{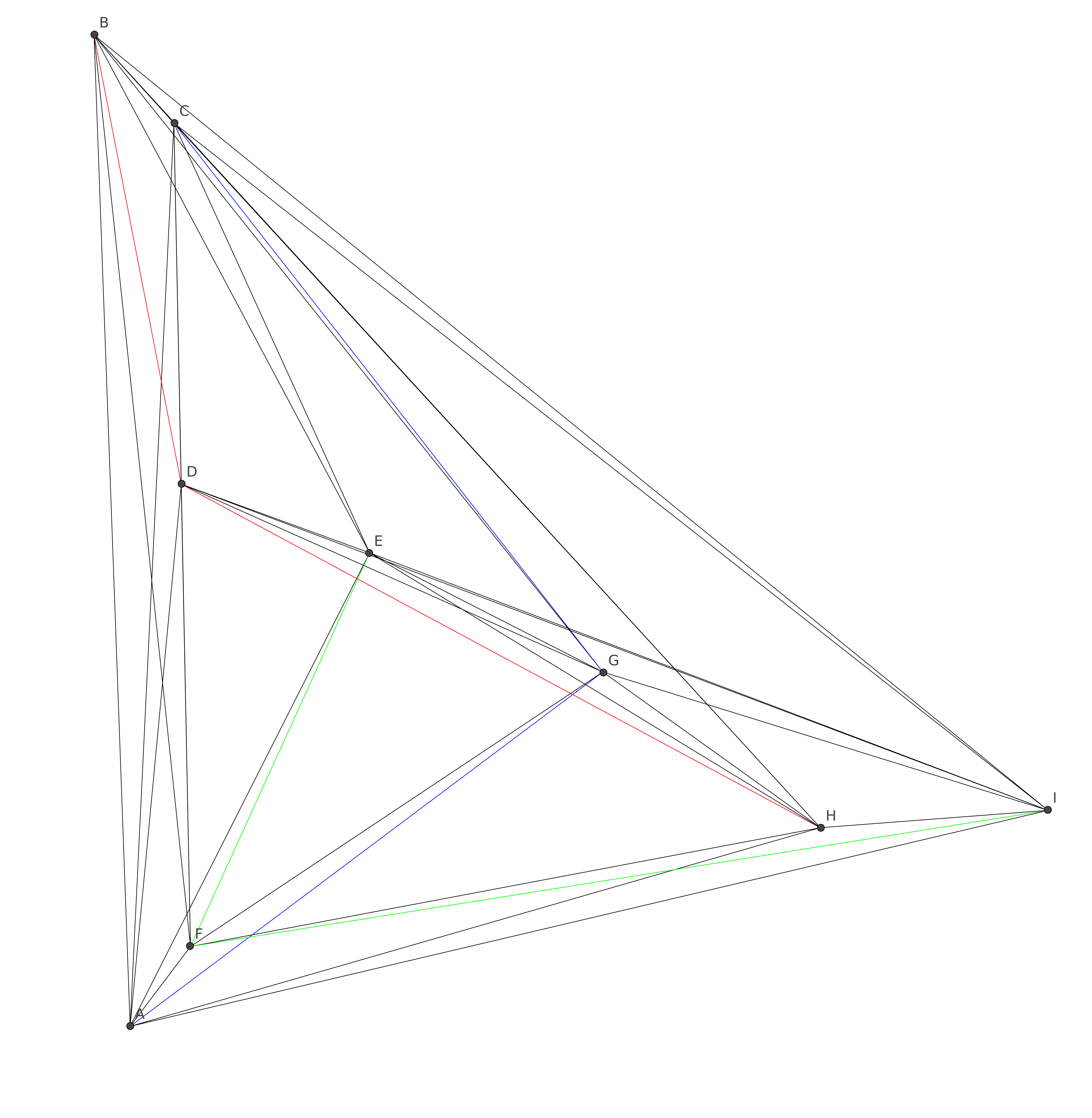}
\end{center}
\end{figure}

\begin{figure}
\begin{center}
\includegraphics[scale=0.23]{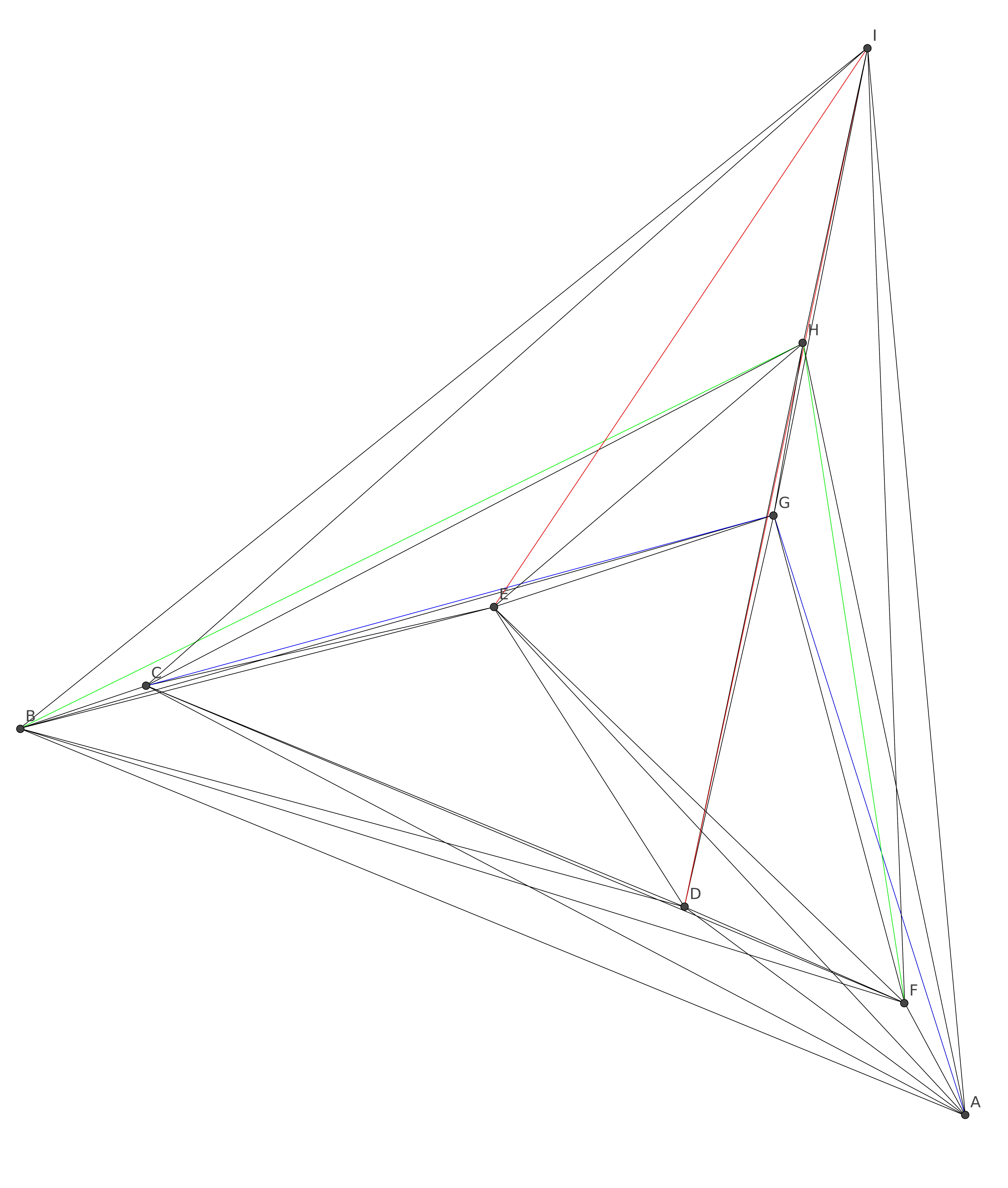}
\end{center}
\end{figure}

\begin{figure}
\begin{center}
\includegraphics[scale=0.23]{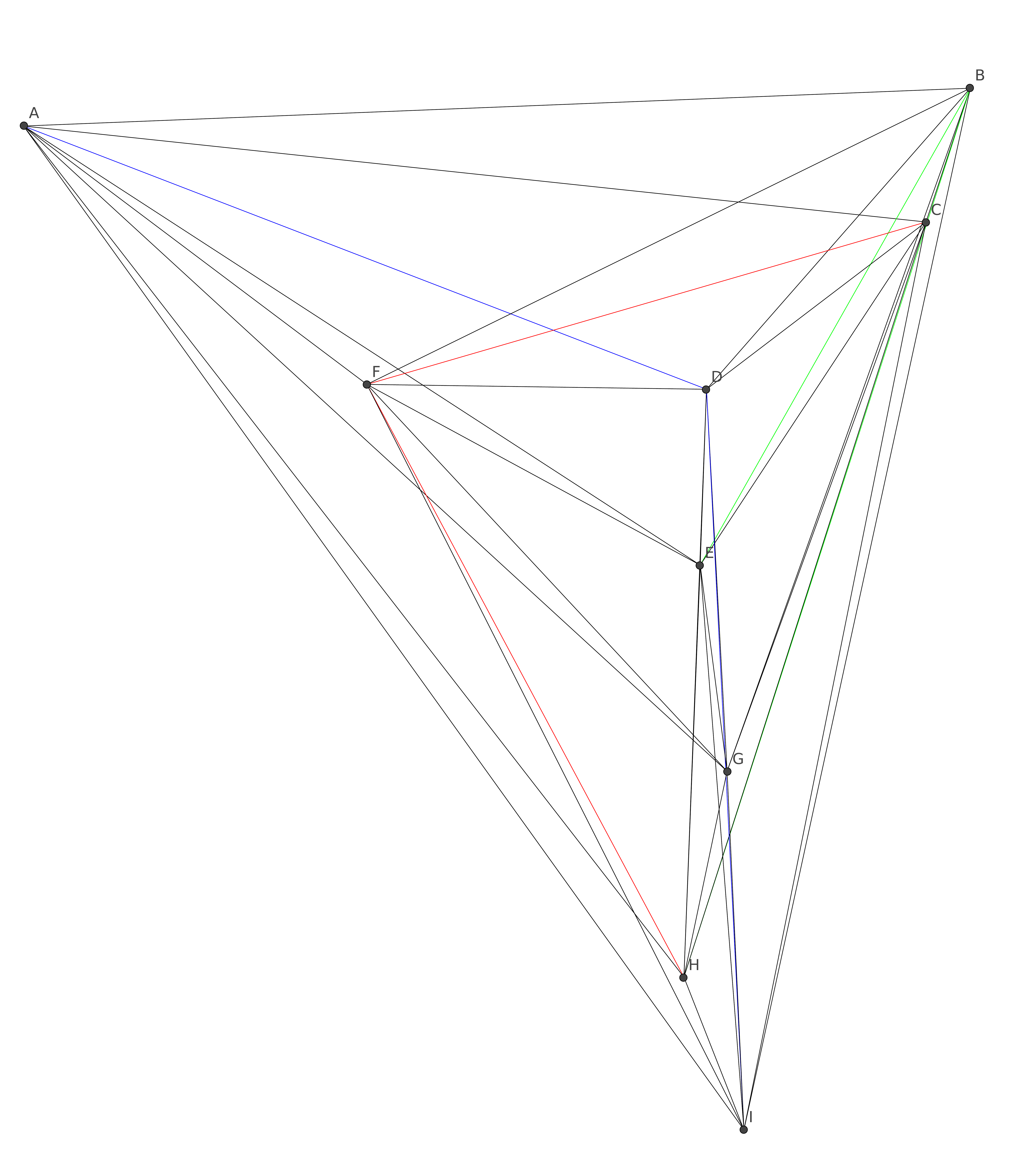}
\end{center}
\end{figure}

\begin{figure}
\begin{center}
\includegraphics[scale=0.23]{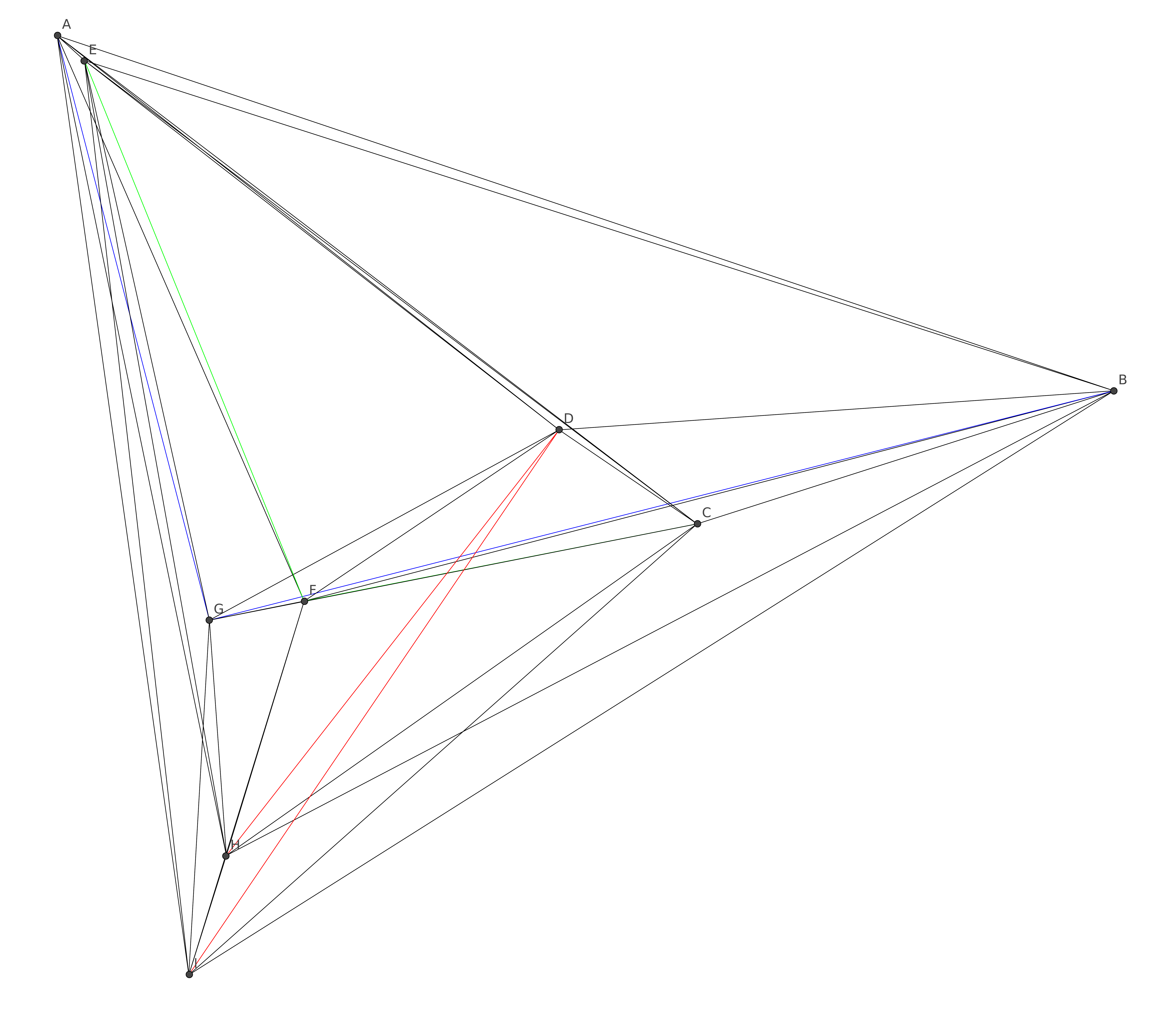}
\end{center}
\end{figure}

\begin{figure}
\begin{center}
\includegraphics[scale=0.23]{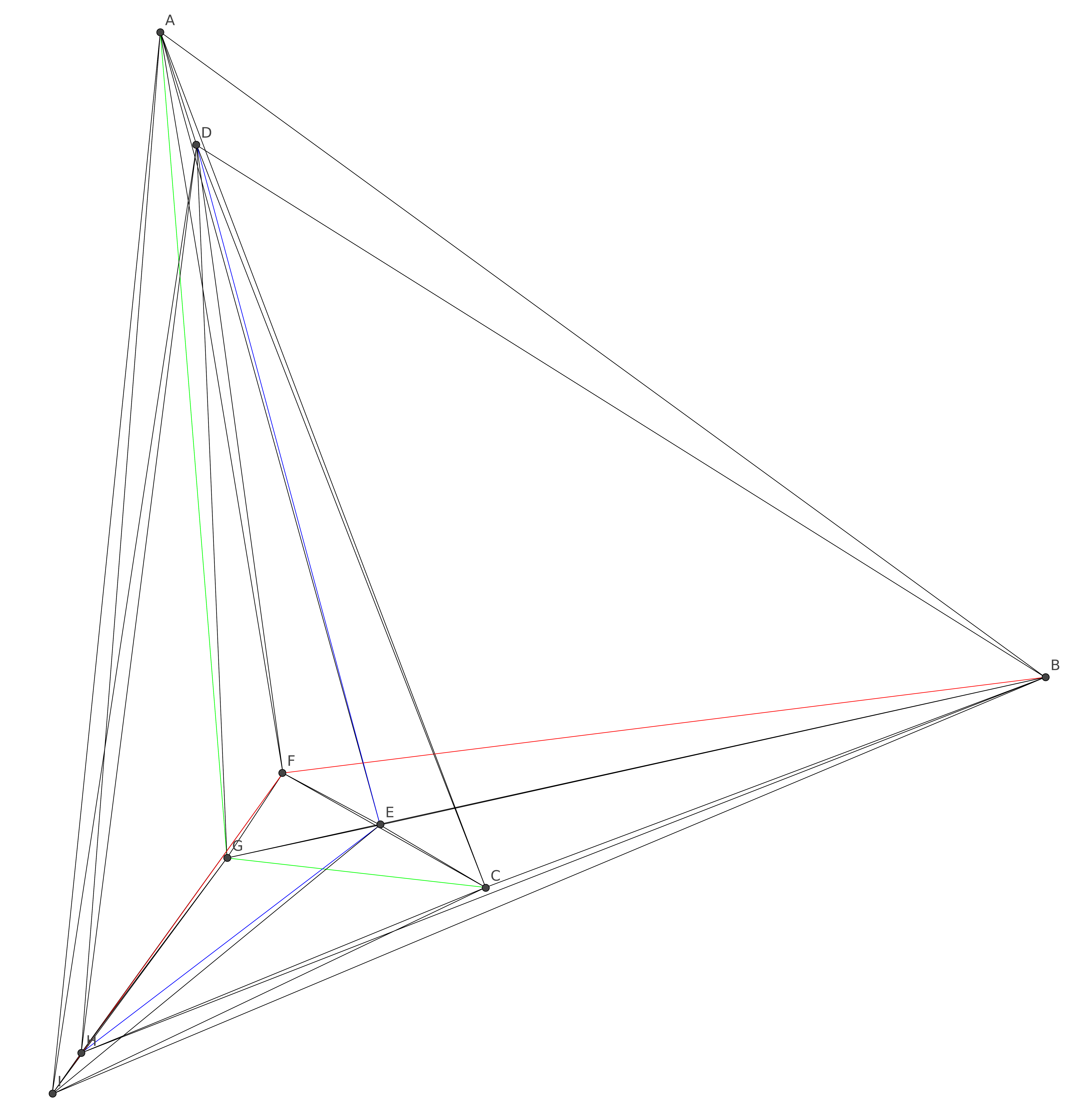}
\end{center}
\end{figure}

\begin{figure}
\begin{center}
\includegraphics[scale=0.5]{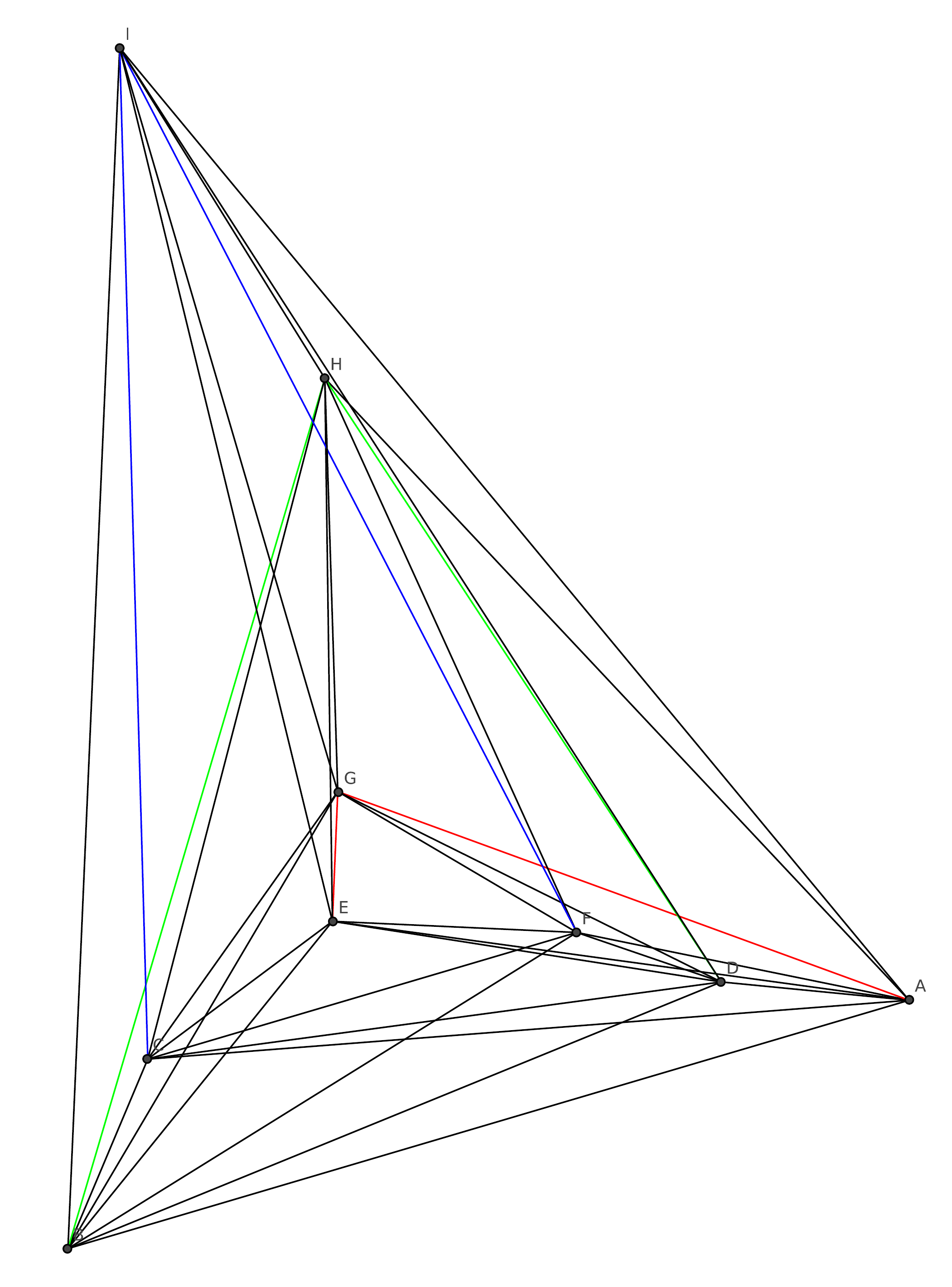}
\end{center}
\end{figure}


\end{document}